%% file: TranDinh_AMA_recovery_2015.tex
\RequirePackage{fix-cm}
\documentclass{svjour3}                     
\smartqed  
\usepackage{graphicx}
\input{TranDinh_OptLetter2015_defs}

\begin{document}

\title{Construction and Iteration-Complexity of Primal Sequences in Alternating Minimization Algorithms
}

\titlerunning{Primal Solution Recovery in Alternating Minimization Algorithms}        

\author{Quoc Tran-Dinh
}


\institute{Quoc Tran-Dinh \at
              Department of Statistics and Operations Research\\
              The University of North Carolina at Chapel Hill, USA \\
              \email{quoctd@email.unc.edu}           
}

\date{Received: date / Accepted: date}

\maketitle
\input{TranDinh_OptLetter2015_abstract}
\input{TranDinh_OptLetter2015_intro}

\input{TranDinh_OptLetter2015_prelim}
\input{TranDinh_OptLetter2015_algs}
\input{TranDinh_OptLetter2015_concl}
\input{TranDinh_OptLetter2015_appendix}



\vspace{-1ex}
\bibliographystyle{spmpsci}      

\input{TranDinh_OptLetter2015_refs}

\end{document}

%% file: TranDinh_OptLetter2015_defs.tex
\usepackage{amsmath}
\usepackage{amssymb}
\usepackage{graphicx}
\usepackage{color}
\usepackage{ifpdf}
\usepackage{url}
\usepackage{hyperref}
\usepackage{algorithm}
\usepackage{algorithmic}

\newtheorem{assumption}{Assumption A.\!\!}

\newcommand{\R}{\mathbb{R}}
\newcommand{\Rext}{\mathbb{R}\cup\left\{+\infty\right\}}

\newcommand{\set}[1]{\left\{#1\right\}}
\newcommand{\norm}[1]{\Vert#1\Vert}
\newcommand{\Eproof}{\hfill $\square$}
\newcommand{\argmin}{\mathrm{arg}\!\min}
\newcommand{\dom}[1]{\mathrm{dom}(#1)}
\newcommand{\xb}{\mathbf{x}}

\newcommand{\zb}{\mathbf{z}}
\newcommand{\ub}{\mathbf{u}}
\newcommand{\vb}{\mathbf{v}}

\newcommand{\cb}{\mathbf{c}}
\newcommand{\Ab}{\mathbf{A}}
\newcommand{\Bb}{\mathbf{B}}

\newcommand{\rb}{\mathbf{r}}
\renewcommand{\sb}{\mathbf{s}}

\newcommand{\iprods}[1]{\langle #1\rangle}
\newcommand{\lbd}{\boldsymbol{\lambda}}

\newcommand{\Uc}{\mathcal{U}}
\newcommand{\Vc}{\mathcal{V}}
\newcommand{\Xc}{\mathcal{X}}

\newcommand{\Dc}{\mathcal{D}}
\newcommand{\Lc}{\mathcal{L}}

\newcommand{\fopt}{f^{\star}}
\newcommand{\uopt}{\mathbf{u}^{\star}}
\newcommand{\vopt}{\mathbf{v}^{\star}}
\newcommand{\xopt}{\mathbf{x}^{\star}}

\renewcommand{\vopt}{\mathbf{v}^{\star}}
\newcommand{\dopt}{d^{\star}}
\newcommand{\lbds}{\boldsymbol{\lambda}^{\star}}
\newcommand{\uhat}{\hat{\mathbf{u}}}
\newcommand{\vhat}{\hat{\mathbf{v}}}

\newcommand{\ut}{\tilde{\mathbf{u}}}
\newcommand{\vt}{\tilde{\mathbf{v}}}
\newcommand{\xt}{\tilde{\mathbf{x}}}
\newcommand{\ubar}{\bar{\mathbf{u}}}
\newcommand{\vbar}{\bar{\mathbf{v}}}
\newcommand{\xbar}{\bar{\mathbf{x}}}

\newcommand{\lbdh}{\hat{\boldsymbol{\lambda}}}

\newcommand{\prox}{\mathrm{prox}}

%% file: TranDinh_OptLetter2015_abstract.tex
\vspace{-2ex}
\begin{abstract}
We introduce a new weighted averaging scheme using  ``Fenchel-type'' operators to recover primal solutions in the alternating minimization-type algorithm (AMA) for prototype constrained convex optimization.
Our approach combines the classical AMA idea in \cite{Tseng1991} and  Nesterov's prox-function smoothing technique without requiring the strong convexity of the objective function.
We develop a new non-accelerated primal-dual AMA method and estimate its primal convergence rate  both on the objective residual and on the feasibility gap.
Then, we incorporate Nesterov's accelerated step into this algorithm and obtain a new accelerated primal-dual AMA variant endowed with a rigorous convergence rate guarantee.
We show that the worst-case iteration-complexity in this algorithm is optimal (in the sense of first-oder black-box models),  without imposing the full strong convexity assumption on the objective.

\keywords{Alternating minimization algorithm \and smoothing technique \and primal solution recovery \and accelerated first-oder method \and constrained convex optimization}
\end{abstract}

%% file: TranDinh_OptLetter2015_intro.tex
\vspace{-2ex}
\section{Introduction}\label{sec:intro}
\vspace{-2ex}
This paper studies a new weighted-averaging strategy  in alternating minimization-type algorithms (AMA)  to recover a primal solution of the following  constrained convex optimization problem:
\begin{equation}\label{eq:constr_cvx}
\fopt := \left\{\begin{array}{ll}
\displaystyle\min_{\ub,\vb} & \set{ f(\xb) := g(\ub) + h(\vb) }\\
\mathrm{s.t.} & \Ab\ub + \Bb\vb = \cb, ~~\ub\in\Uc, ~\vb\in\Vc,
\end{array}\right.
\end{equation}
where $g : \R^{p_1}\to\Rext$ and $h : \R^{p_2}\to\Rext$ are both proper, closed and convex (not necessarily strongly convex),  ($p_1 + p_2 = p$, $\Ab\in\R^{n\times p_1}$, $\Bb\in\R^{n\times p_2}$,  $\cb\in\R^n$, and $\Uc\subset\R^{p_1}$ and $\Vc\subset\R^{p_2}$ are two nonempty, closed and convex sets. 

Problem~\eqref{eq:constr_cvx} surprisingly  covers a broad class of constrained convex  programs, including composite convex minimization, general linear constrained convex optimization problems, and conic programs.

Primal-dual methods handle problem \eqref{eq:constr_cvx} together with its dual formulation, and generate a primal-dual sequence so that it converges to a primal and dual solution of \eqref{eq:constr_cvx}.
Research on primal-dual methods has been extensively studied in the literature for many decades, see, e.g., \cite{Bertsekas1996d,Rockafellar1970,Wright1997} and the references quoted therein.
However, such methods have attracted a great attention in the past decade due to new applications in signal and image processing, economics, machine learning, and statistics.
Various primal-dual methods have been rediscovered and extended, not only from algorithmic perspectives, but also from theoretical convergence guarantees.
Despite of this great attempt in the algorithmic development,  the corresponding supporting theory has not been well-developed, especially, the algorithms with rigorous convergence guarantees and low complexity-per-iteration.

Perhaps, applying first order methods to the dual is the most nature approach to solve constrained problems of the form \eqref{eq:constr_cvx}.
By means of the Lagrange duality theory, we can formulate the dual problem of \eqref{eq:constr_cvx} as a convex problem, where existing convex optimization techniques can be applied to solve it.
Depending on the structure assumptions imposing on \eqref{eq:constr_cvx}, the dual problem possesses useful properties that can be exploited to develop algorithms for the dual. 
For instance, we can use subgradient, gradient, proximal-gradient, as well as other proximal and splitting techniques to solve this problem. 
Then, the primal solutions of \eqref{eq:constr_cvx} can be recovered from the dual solutions \cite{necoara2014iteration,Yurtsever2015}.
Among many other primal-dual splitting methods, alternating minimization algorithm (AMA) proposed by Tseng \cite{Tseng1991} becomes one of the most popular and powerful methods to solve \eqref{eq:constr_cvx} when $g$ and $h$ are nonsmooth and convex, and either $g$ or $h$ is strongly convex.
Unfortunately, to the best of our knowledge, there has existed no optimization scheme to recover primal solutions of \eqref{eq:constr_cvx} in AMAs with convergence rate guarantees on both the primal objective residual and the feasibility gap.

If $g$ and $h$ are nonsmooth, then numerical methods for solving \eqref{eq:constr_cvx} often rely on the proximal operators of $g$ and $h$. 
Mathematically, a proximal operator of a proper, closed, and convex function $\varphi : \R^p\to\Rext$ is defined as:
\vspace{-1ex}
\begin{equation}\label{eq:prox}
\prox_{\varphi}(\xb) := \argmin_{\zb}\set{ \varphi(\zb) + (1/2)\norm{\zb - \xb}^2}.
\vspace{-1ex}
\end{equation}
If $\prox_{\varphi}$ can be computed efficiently, i.e., by a closed form or by a polynomial time algorithm, then we say that $\varphi$ has a ``tractable proximity'' operator.
There exist many smooth and nonsmooth convex functions with tractable proximity operators as indicated in, e.g., \cite{Combettes2011,Parikh2013}.
The proximal operator is in fact a special case of the resolvent  in monotone inclusions \cite{Rockafellar1976b}.
Principally, the optimality condition for  \eqref{eq:constr_cvx} can be cast into a monotone inclusion \cite{Bauschke2011,Facchinei2003}.
By mean of proximity operators and gradients, splitting approaches in monotone inclusions can be applied to solve such a problem \cite{Combettes2014,Chambolle2011,Facchinei2003}.
However, due to this generalization, the convergence guarantees and the convergence rates of these algorithms often achieve via a primal-dual gap or residual metric joined both the primal and dual variables.
Such convergence guarantees do not reveal the complexity bounds of the primal sequence for \eqref{eq:constr_cvx} at intermediate iterations when we terminate the algorithm at a desired accuracy.

Our approach in this paper is briefly described as follows.
First, since we work with non-strongly convex objectives $g$ and $h$, we employ Nesterov's smoothing technique via prox-functions \cite{Nesterov2005c} to partially smooth the dual function.
Then, we apply the forward-backward splitting method to solve the smoothed dual problem, which is exactly the AMA method in \cite{Tseng1991}.
Next, we introduce a new weighted averaging scheme using  the Fenchel-type operators (c.f. \eqref{eq:sharp_oper}) to generate the primal sequence simultaneously with the dual one. 
We then prove convergence rate guarantees for \eqref{eq:constr_cvx} in the primal variable as opposed to the dual one as in \cite{Goldstein2012}. 
Finally, by incorporating Nesterov's acceleration step into the forward-backward splitting method, we obtain an accelerated primal-dual variant for solving \eqref{eq:constr_cvx} with a primal convergence rate guarantee.
Interestingly, we can show that the primal sequence converges to an optimal solution of \eqref{eq:constr_cvx} with the $\mathcal{O}(1/k^2)$-optimal rate provided that only $g$ or $h$ is strongly convex, but not the whole function $f$ as in accelerated dual gradient methods \cite{necoara2014iteration}, where $k$ is the iteration counter.

\vspace{-2ex}
\paragraph{Our contributions:} Our specific contributions can be summarized as follows:
\vspace{-1.2ex}
\begin{itemize}
\item[$\mathrm{a)}$] We propose to combine Nesterov's smoothing technique, the alternating minimization idea, and the weighted-averaging strategy to develop a new primal-dual AMA algorithm for solving \eqref{eq:constr_cvx} without strong convexity assumption on $g$ or $h$.
We characterize the convergence rate on the absolute primal objective residual $\vert f(\xbar^k) - \fopt\vert$ and feasibility gap $\Vert\Ab\bar{\ub}^k + \Bb\bar{\vb}^k - \cb\Vert$ for the averaging primal sequence $\set{\xbar^k}$.
By an appropriate choice of the smoothness parameter, we  provide the worst-case iteration-complexity of this algorithm to obtain an $\epsilon$-primal solution. 

\item[$\mathrm{b)}$] By incorperatiing  Nesterov's accelerated step, we develop a new accelerated primal-dual AMA variant for solving \eqref{eq:constr_cvx}, and characterize its worst-case iteration-complexity which  is optimal in the sense of first-oder black-box models \cite{Nemirovskii1983}.

\item[$\mathrm{c)}$] When either $g$ or $h$ is strongly convex, we recover the standard AMA algorithm as in  \cite{Goldstein2012}, but with our averaging strategy, we obtain the $\mathcal{O}(1/k^2)$-convergence rate  on $\vert f(\xbar^k) - \fopt\vert$ and  $\Vert\Ab\bar{\ub}^k + \Bb\bar{\vb}^k - \cb\Vert$ separably for the primal problem \eqref{eq:constr_cvx}, not for its dual.
\end{itemize}
\vspace{-1ex}
Let us emphasize the following points of our contributions.
First, we can view the algorithms presented in this paper as the ISTA and FISTA schemes \cite{Beck2009} applied to the smoothed dual problem of \eqref{eq:constr_cvx} instead the original dual of \eqref{eq:constr_cvx} as in \cite{Goldstein2012}.
The convergence rate on the dual objective residual is well-known and standard, while the convergence rates on the primal sequence are new.
Second, we adapt the weights  in our averaging primal sequence (c.f. \eqref{eq:averaging_sequence}) to the local Lipschitz constant via a back-tracking line-search, which  potentially increases the empirical performance of the algorithms.
Third, the averaging primal sequence is computed via an additional sharp-operator of $h_{\Vc}$ (c.f. \eqref{eq:sharp_oper}) instead of the current primal iterate. This computation can be done efficiently (e.g., in a closed form) when $h_{\Vc}$ has a decomposable structure.

\vspace{-3ex}
\paragraph{Paper organization:}
The rest of this paper is organized as follows.
Section \ref{sec:pd_framework} briefly describes standard Lagrange duality framework for \eqref{eq:constr_cvx}, and shows how to apply Nesterov's smoothing idea to the dual problem.
The main results are presented in Sections \ref{sec:ama_algs} and \ref{sec:ac_ama}, where the two new algorithms and their convergence are provided.
Section \ref{sec:strong_convexity_case} is devoted to investigating the strongly convex case.
Concluding remarks are given in Section \ref{sec:concluding}, while  technical proof   is moved to the appendix.

%% file: TranDinh_OptLetter2015_prelim.tex
\vspace{-4ex}
\section{Primal-dual framework and smoothing technique}\label{sec:pd_framework}
\vspace{-2.5ex}
First, we briefly present the Lagrange duality framework for \eqref{eq:constr_cvx}. 
Then we show how to apply Nesterov's smoothing technique to smooth the dual function of \eqref{eq:constr_cvx}.

\vspace{-3ex}
\subsection{The Lagrange primal-dual framework}
\vspace{-2ex}
Let $\xb := (\ub, \vb)$ denote the primal variables, and $\Dc := \set{ \xb \in \Uc\times\Vc : \Ab\ub + \Bb\vb = \cb}$ denote the feasible set of \eqref{eq:constr_cvx}.
We define the Lagrange function of \eqref{eq:constr_cvx} corresponding to the linear constraint $\Ab\ub + \Bb\vb = \cb$ as $\Lc(\xb,  \lbd) := g(\ub) + h(\vb) + \iprods{\lbd, \cb - \Ab\ub - \Bb\vb}$, where $\lbd$ is the vector of Lagrange multipliers. 
Then, we can define the dual function $d$ of \eqref{eq:constr_cvx} as
\begin{equation}\label{eq:dual_func}
d(\lbd) := \min_{\ub\in\Uc,\vb\in\Vc}\set{ g(\ub) + h(\vb) + \iprods{\lbd, \cb - \Ab\ub - \Bb\vb}}.
\end{equation}
Clearly, $d$ can be split into three terms $d(\lbd) = d^1(\lbd) + d^2(\lbd) + \iprods{\cb,\lbd}$, where
\begin{equation}\label{eq:dual_func_di}
\left\{\begin{array}{ll}
d^1(\lbd) &:= \displaystyle\min_{\ub\in\Uc}\set{ g(\ub) - \iprods{\Ab^T\lbd, \ub} },\vspace{0.75ex}\\
d^2(\lbd) &:= \displaystyle\min_{\vb\in\Vc}\set{ h(\vb) - \iprods{\Bb^T\lbd, \vb} }.
\end{array}\right.
\end{equation}
Using $d$, we can define the dual problem of \eqref{eq:constr_cvx} as
\begin{equation}\label{eq:dual_prob}
\dopt := \max_{\lbd\in\R^n}d(\lbd).
\end{equation}
We say that problem \eqref{eq:constr_cvx} satisfies the Slater condition if
\begin{equation}\label{eq:slater_cond}
\mathrm{ri}(\Xc) \cap \set{\Ab\ub + \Bb\vb =  \cb} \neq\emptyset,
\end{equation}
where $\Xc := \Uc\times \Vc$ and $\mathrm{ri}(\Xc)$ is a the relative interior of $\Xc$ \cite{Rockafellar1970}.

In this paper, we require the following blanket assumptions, which are  standard in convex optimization.

\vspace{-1ex}
\begin{assumption}\label{as:A1}
The functions $g$ and $h$ are both proper, closed, and convex $($not necessarily strongly convex$)$. 
The solution set $\Xc^{\star}$ of \eqref{eq:constr_cvx} is nonempty.
The Slater condition \eqref{eq:slater_cond} holds for \eqref{eq:constr_cvx}.
\end{assumption}
\vspace{-1ex}

It is well-known that, under Assumption A.\ref{as:A1},  strong duality in \eqref{eq:constr_cvx} and \eqref{eq:dual_prob} holds, i.e., we have zero duality gap which is expressed as $\fopt - \dopt = 0$. 
Moreover, for any feasible point $(\xb, \lbd)\in\dom{f}\times\R^n$ and any primal-dual solution $(\xopt,\lbds)$ with $\xopt := (\uopt,\vopt) \in \Xc^{\star}$ we have: $\Lc(\xopt,\lbd) \leq \Lc(\xopt, \lbds) = \fopt = \dopt \leq \Lc(\xb,\lbds)$ for all $\xb\in\Xc$ and $\lbd\in\R^n$.

Now, let us consider the components $d^1$ and $d^2$ of \eqref{eq:dual_func_di}. 
Indeed, we can write these components as
\begin{equation*}
\begin{array}{ll}
d^1(\lbd) &= -\displaystyle\max_{\ub\in\Uc}\set{\iprods{\Ab^T\lbd,\ub} - g(\ub)} = -g_{\Uc}^{*}(\Ab^T\lbd),\\
d^2(\lbd) &= -\displaystyle\max_{\vb\in\Vc}\set{\iprods{\Bb^T\lbd,\vb} - h(\vb)} = -h_{\Vc}^{*}(\Bb^T\lbd),\\
\end{array}
\end{equation*}
where $g_{\Uc}^{*}$ and $h_{\Vc}^{*}$ are the Fenchel conjugate of $g_{\Uc} := g + \delta_{\Uc}$ and $h_{\Vc} := h + \delta_{\Vc}$, respectively \cite{Rockafellar1970}.
If we define two multivalued maps
\begin{equation}\label{eq:sharp_oper}
\ub^{\#}(\sb) := \arg\!\max_{\ub\in\Uc}\set{\iprods{\sb, \ub} - g(\ub)},~\text{and}~\vb^{\#}(\sb) := \arg\!\max_{\vb\in\Vc}\set{\iprods{\sb, \vb} - h(\vb)},
\end{equation}
then the solution $\ub^{\ast}(\lbd)$ of $d^1$ in \eqref{eq:dual_func_di} is given by $\ub^{\ast}(\lbd) \in \ub^{\#}(\Ab^T\lbd) \equiv \partial{g_{\Uc}^{*}}(\Ab^T\lbd)$. 
Similarly, the solution $\vb^{\ast}(\lbd)$ of $d^2$ in \eqref{eq:dual_func_di} is given by $\vb^{\ast}(\lbd) \in \vb^{\#}(\Bb^T\lbd) \equiv \partial{h_{\Vc}^{*}}(\Bb^T\lbd)$.
We call $\ub^{\#}$ and $\vb^{\#}$ the \textit{sharp}-operator of $g$ and $h$, respectively \cite{Yurtsever2015}.
Each \textit{oracle call} to $d$ queries one element of the \textit{sharp}-operators $\ub^{\#}$ and $\vb^{\#}$ at a given $\lbd\in\R^n$.

By using the saddle point relation, we can show that $f^{*} \leq \Lc(\xb, \lbds) = f(\xb) - \iprods{\Ab\ub + \Bb\vb - \cb, \lbds} \leq f(\xb) +  \norm{\Ab\ub + \Bb\vb - \cb}\norm{\lbds}$ for any $\xb\in\Xc$. Hence, we have
\begin{equation}\label{eq:lower_bound}
-\norm{\lbds}\norm{\Ab\ub + \Bb\vb - \cb} \leq f(\xb) - \fopt \leq f(\xb) - d(\lbd).
\end{equation}
In this paper, we only assume that the second dual component $d^2$ defined by \eqref{eq:dual_func_di} satisfies the following assumption.

\vspace{-1ex}
\begin{assumption}\label{as:A1b}
The dual component $d^2$ defined by \eqref{eq:dual_func_di} is finite.
\end{assumption}
\vspace{-1ex}

\noindent This assumption holds in particular when $\Vc$ is bounded. Moreover, $\vb^{*}(\lbd)$ is well-defined for any $\lbd\in\R^n$.
Throughout this paper, we assume that Assumptions A.\ref{as:A1} and A.\ref{as:A1b} holds without referring to them again.

\vspace{-3.5ex}
\subsection{The primal weighted averaging sequence}
\vspace{-2ex}
Given a sequence of the primal approximation $\set{\tilde{\xb}^k}_{k\geq 0}$, where $\tilde{\xb}^k := (\tilde{\ub}^k,\tilde{\vb}^k) \in \Xc$. 
We define the following weighted averaging sequence $\set{\xbar^k}$ with $\xbar^k := (\bar{\ub}^k, \bar{\vb}^k)$ as
\begin{equation}\label{eq:averaging_sequence}
\bar{\ub}^k := S_k^{-1}\sum_{i=1}^kw_i\tilde{\ub}^i, ~~~~~~~~\bar{\vb}^k := S_k^{-1}\sum_{i=0}^kw_i\tilde{\vb}^i, ~~~~ \text{and}~~~~S_k := \sum_{i=0}^kw_i,
\end{equation}
where $\set{w_i}_{i\geq0}\subset\R_{++}$ is the corresponding weights.

To avoid storing the whole sequence $\set{\tilde{\ub}^k, \tilde{\vb}^k)}$ in our algorithms, we can compute $\set{\bar{\xb}^k}$ recursively as follows:
\begin{equation}\label{eq:averaging_sequence1}
\bar{\ub}^k := (1-\tau_k)\bar{\ub}^{k-1} + \tau_k\tilde{\ub}^k, ~~\text{and}~~~\bar{\vb}^k := (1-\tau_k)\bar{\vb}^{k-1} + \tau_k\tilde{\vb}^k, ~~\forall k\geq 1,
\end{equation}
where $\tau_k := \frac{w_k}{S_k} \in [0, 1]$, $\ubar^0 := \tilde{\ub}^0$, and $\vbar^0 := \tilde{\vb}^0$. 
Clearly, for any convex function $f$, we have $f(\xbar^k) \leq S_k^{-1}\displaystyle\sum_{i=0}^kw_if(\tilde{\xb}^i)$ by the well-known Jensen  inequality.

\vspace{-3.5ex}
\paragraph{\textbf{Approximate solutions:}}
Our goal is to approximate a solution $\xb^{\star}$ of \eqref{eq:constr_cvx} by $\xopt_{\epsilon}$ in the following sense:

\vspace{-2ex}
\begin{definition}\label{de:approx_sols}
Given an accuracy level $\epsilon > 0$, a point $\xopt_{\epsilon} := (\uopt_{\epsilon}, \vopt_{\epsilon})\in\Xc$ is said to be an $\epsilon$-solution of \eqref{eq:constr_cvx} if
\begin{equation}\label{eq:approx_sols}
\vert f(\xopt_{\epsilon}) - \fopt\vert \leq \epsilon ~~\text{and}~~~~ \norm{\Ab\uopt_{\epsilon} + \Bb\vopt_{\epsilon} - \cb}  \leq \epsilon.
\end{equation}
\end{definition}
\vspace{-1ex}
Here, we call $\vert f(\xopt_{\epsilon}) - \fopt\vert$ the [absolute] primal objective residual and $\norm{\Ab\uopt_{\epsilon} + \Bb\vopt_{\epsilon} - \cb}$ the primal feasibility gap. 
The condition $\xopt_{\epsilon} \in\Xc$ is in general not restrictive since, in many cases, $\Xc$ is a simple set (e.g., a box, a simplex, or a conic cone) so that the projection onto $\Xc$ can exactly be guaranteed. 

\vspace{-3.5ex}
\subsection{Smoothing the dual component}
\vspace{-2ex}
As mentioned earlier, we first focus on the non-strongly convex functions $g$ and $h$.
In this case, we can not directly apply the standard AMA \cite{Tseng1991} to solve \eqref{eq:constr_cvx}.
We smooth $g$ by using a prox-function as follows.

A continuous and strongly convex function $p_{\Uc}$ with the strong convexity parameter $\mu_p > 0$ is called a prox-function for $\Uc$ if $\Uc\subseteq \dom{p_{\Uc}}$ \cite{Nesterov2005c}.
We consider the following smoothed function $d^1_{\gamma}$ for $d^1$:
\begin{equation}\label{eq:d1_gamma}
d_{\gamma}^1(\lbd) := \min_{\ub\in\Uc}\set{ g(\ub) - \iprods{\lbd, \Ab\ub} + \gamma p_{\Uc}(\ub)},
\end{equation}
where $\gamma > 0$ is a smoothness parameter.

It is well-known that $d^1_{\gamma}$ is concave and smooth. 
Moreover, as shown in \cite{Nesterov2005c}, its gradient is given by $\nabla{d^1_{\gamma}}(\lbd) = -\Ab\ub_{\gamma}^{*}(\lbd)$, which is Lipschitz continuous with the Lipschitz constant $L^{\gamma}_{d^1} := \frac{\norm{\Ab}^2}{\gamma\mu_p}$, where $\ub_{\gamma}^{*}(\lbd)$ is the unique solution of the minimization problem in \eqref{eq:d1_gamma}.
In addition, we have the following estimate
\begin{equation}\label{eq:d1_gamma_est}
d_{\gamma}^1(\lbd) - \gamma D_{\Uc} \leq d^1(\lbd) \leq d^1_{\gamma}(\lbd), ~~\forall~\lbd\in\R^n,
\end{equation}
where $D_{\Uc}$ is the prox-diameter of $\Uc$, i.e.,
\begin{equation}\label{eq:DU}
D_{\Uc} := \sup_{\ub\in\Uc}p_{\Uc}(\ub).
\end{equation}
In order to develop algorithms, we require the following additional assumption.

\begin{assumption}\label{as:A4}
The quantity $D_{\Uc}$ defined by \eqref{eq:DU} is finite, i.e., $0 \leq D_{\Uc} < +\infty$.
\end{assumption}

Clearly, if $\Uc$ is bounded, then Assumption A.\ref{as:A4} is automatically satisfied. 
Under Assumption A.\ref{as:A4}, we consider the following convex problem:
\begin{equation}\label{eq:smoothed_dual_prob}
d_{\gamma}^{\star} := \max_{\lbd\in\R^n}\set{ d_{\gamma}(\lbd) := d_{\gamma}^1(\lbd) + d^2(\lbd) + \iprods{\cb,\lbd} }.
\end{equation}
Using \eqref{eq:d1_gamma_est}, we can see that $d^{\star}_{\gamma}$ converges to $d^{\star}$ as $\gamma \downarrow 0^{+}$.
Hence, \eqref{eq:smoothed_dual_prob} can be considered as an approximation to the dual problem \eqref{eq:dual_prob}.
We call \eqref{eq:smoothed_dual_prob} the \textit{smoothed} dual problem of \eqref{eq:constr_cvx}.

%% file: TranDinh_OptLetter2015_algs.tex
\vspace{-3ex}
\section{The non-accelerated primal-dual alternating minimization algorithm}\label{sec:ama_algs}
\vspace{-2ex}
Since $d_{\gamma}^1$ is Lipschitz gradient, we can apply the proximal-gradient method (ISTA \cite{Beck2009}) to solve \eqref{eq:smoothed_dual_prob}. This leads to the AMA scheme presented in \cite{Goldstein2012,Tseng1991}. 

The main iteration of the alternating minimization algorithm (AMA) \cite{Tseng1991} applying to the corresponding primal problem of \eqref{eq:smoothed_dual_prob} can be written as
\begin{equation}\label{eq:ama_scheme}
\left\{\begin{array}{ll}
\hat{\ub}^{k+1} &:= \displaystyle\argmin_{\ub\in\Uc}\big\{ g(\ub) - \iprods{\Ab^T\lbdh^k, \ub} + \gamma p_{\Uc}(\ub)\big\} = \nabla{g_{\gamma}^{*}}(\Ab^T\lbdh^k),\\
\hat{\vb}^{k+1} &:= \displaystyle\argmin_{\vb\in\Vc}\big\{ h(\vb) - \iprods{\Bb^T\lbdh^k, \vb} + \frac{\eta_k}{2}\norm{\cb - \Ab\hat{\ub}^{k+1} - \Bb\vb}^2\big\},\\
\lbd^{k+1} &:= \lbdh^k + \eta_k(\cb - \Ab\hat{\ub}^{k+1} - \Bb\hat{\vb}^{k+1}),
\end{array}\right.
\end{equation}
where $\lbdh^k\in\R^n$ is given, $\eta_k > 0$ is the penalty parameter, and $g_{\gamma}(\cdot) := g(\cdot) + \gamma p_{\Uc}(\cdot)$.
We define the quadratic surrogate of $d^1$ as follows:
\begin{equation}\label{eq:Q_surrogate}
Q^{\gamma}_{L_k}(\lbd;\lbdh^k) := d_{\gamma}^1(\lbdh^k) + \iprods{\nabla{d_{\gamma}^1}(\lbdh^k),\lbd - \lbdh^k} - \frac{L_k}{2}\norm{\lbd - \lbdh^k}^2.
\end{equation}
Then the following lemma provides a key estimate to prove the convergence of the algorithms in the sequel, whose proof can be found in Appendix \ref{apdx:le:upper_surrogate}.

\begin{lemma}\label{le:upper_surrogate}
The smoothed dual component $d^1_{\gamma}$ defined by \eqref{eq:d1_gamma} is concave and smooth. It satisfies the following estimate
\begin{equation}\label{eq:quad_approx}
d^1_{\gamma}(\lbd) + \iprods{\nabla{d^1_{\gamma}}(\lbd),\tilde{\lbd} - \lbd} - \frac{L_{d^1}}{2}\norm{\tilde{\lbd} - \lbd}^2 \leq d^1(\tilde{\lbd}), ~~\forall\lbd,\tilde{\lbd}\in\R^n,
\end{equation}
where $L^{\gamma}_{d^1} := \frac{\Vert\Ab\Vert^2}{\gamma\mu_p}$.

Let  $\lbd^{k+1}$ be the point generated by \eqref{eq:ama_scheme} from $\hat{\lbd}^k$ and $\eta_k$. 
Then, \eqref{eq:ama_scheme} is equivalent to the forward-backward splitting scheme applying to the smoothed dual problem \eqref{eq:smoothed_dual_prob}, i.e.,
\begin{equation}\label{eq:forward_backward}
\lbd^{k+1} := \mathrm{prox}_{(-\eta_kd^2)}\left( \lbdh^k + \eta_k\nabla{d}^1_{\gamma}(\lbdh^k)\right).
\end{equation}
In addition, with $Q^{\gamma}_{L_k}$ defined by \eqref{eq:Q_surrogate}, if the following condition holds
\begin{align}\label{eq:linesearch_cond}
d^1_{\gamma}(\lbd^{k+1}) \geq Q^{\gamma}_{L_k}(\lbd^{k+1};\lbdh^k),
\end{align}
then, for any $\lbd\in\R^n$, the following estimates hold
\begin{align}\label{eq:upper_surrogate}
d_{\gamma}(\lbd^{k+1}) &\geq \ell_k^{\gamma}(\lbd) + \frac{1}{\eta_k}\iprods{\lbd^{k+1} - \lbdh^k, \lbdh^k - \lbd} + \left(\frac{1}{\eta_k}-\frac{L_k}{2}\right)\norm{\lbdh^k- \lbd^{k+1}}^2\nonumber\\
&\geq d_{\gamma}(\lbd) +  \frac{1}{\eta_k}\iprods{\lbd^{k+1} - \lbdh^k, \lbdh^k - \lbd} + \left(\frac{1}{\eta_k}-\frac{L_k}{2}\right)\norm{\lbdh^k - \lbd^{k+1}}^2,
\end{align}
where $\ell^{\gamma}_k(\lbd) := d_{\gamma}^1(\lbdh^k) + \iprods{\nabla{d_{\gamma}^1}(\lbdh^k), \lbd - \lbdh^k} + d^2(\lbd^{k+1}) + \iprods{\nabla{d^2}(\lbd^{k+1}),\lbd - \lbd^{k+1}} + \iprods{\cb, \lbd}$, and $\nabla{d^2}(\lbd^{k+1})\in\partial{d^2}(\lbd^{k+1})$ is a subgradient of $d^2$ at $\lbd^{k+1}$.
\end{lemma}

Our next step is to recover an approximate primal solution $\xbar^k := (\bar{\ub}^k, \bar{\vb}^k)$ of \eqref{eq:constr_cvx} using the weighted averaging  scheme \eqref{eq:averaging_sequence}.
Combing this strategy and \eqref{eq:ama_scheme} we can present the new primal-dual AMA algorithm is as in Algorithm \ref{alg:pd_ama} below.

\begin{algorithm}[!ht]\caption{(\textit{Primal-dual alternating minimization algorithm})}\label{alg:pd_ama}
\begin{algorithmic}
\normalsize
   \STATE {\bfseries Initialization:} 
    \STATE{\hspace{2.5ex}}1.~Choose $\gamma := \frac{\epsilon}{2D_{\Uc}}$, and $\underline{L}$ such that $0 < \underline{L} \leq L^{\gamma}_{d_1} := \frac{\norm{\Ab}^2}{\gamma\mu_p}$.
    \STATE{\hspace{2.5ex}}2.~Choose an initial point $\lbd^0 \in \R^n$. 
    \STATE{\hspace{2.5ex}}3.~Set $S_{-1} := 0$, ~$\ubar^{-1} := 0$ and $\vbar^{-1} := 0$. 
    \FOR{$k := 0$ {\bfseries to} $k_{\max}$}
	\STATE 4. Compute $\ut^k = \hat{\ub}^{k+1} = \ub_{\gamma}^{\ast}(\lbd^k)$ defined in \eqref{eq:d1_gamma}.
	\STATE 5.  Choose $\eta_k\in \left(0, \frac{1}{L^{\gamma}_{d_1}}\right]$ and compute
	\begin{equation*}
	 \hat{\vb}^{k+1} := \mathrm{arg}\displaystyle\min_{\vb\in\Vc}\set{ h(\vb) - \iprods{\Bb^T\lbd^k, \vb} + \frac{\eta_k}{2}\norm{\cb - \Ab\ut^k - \Bb\vb}^2 }.
	 \vspace{-2ex}
	 \end{equation*}
	\STATE 6.  Update $\lbd^{k+1} := \lbd^k + \eta_k\left(\cb - \Ab\hat{\ub}^{k+1} - \Bb\hat{\vb}^{k+1}\right)$.
	\STATE 7.  Compute $\vt^k := \vb^{\ast}(\lbd^{k+1}) \in \vb^{\sharp}\left(\Bb^T\lbd^{k+1}\right)$ defined in \eqref{eq:sharp_oper}.
	\STATE 8.  Update $S_k := S_{k-1} + w_k$, with $w_k := \eta_k$, and $\tau_k := \frac{w_k}{S_k}$.
	\STATE 9. Update $\bar{\ub}^{k} := (1-\tau_k)\bar{\ub}^{k-1}  + \tau_k\tilde{\ub}^k$ and $\bar{\vb}^{k} := (1-\tau_k)\bar{\vb}^{k-1}  + \tau_k\tilde{\vb}^k$.
   \ENDFOR
\STATE {\bfseries Output:} The sequence $\set{\xbar^k}$ with $\xbar^k := (\ubar^k, \vbar^k)$.
\end{algorithmic}
\end{algorithm}
In fact, we can use the Lipschitz constant $L^{\gamma}_{d^1} = \frac{\norm{\Ab}^1}{\gamma\mu_p}$ to compute the constant step $\eta_k$ as $\eta_k := \frac{1}{L^{\gamma}_{d^1}}$ at Step 5.
However, we can adaptively choose $\eta_k = L_k^{-1}$ via a back-tracking line-search procedure in Algorithm~\ref{alg:pd_ama} to guarantee the condition \eqref{eq:linesearch_cond}, and this usually performs better in practice than the constant step-size.

Algorithm~\ref{alg:pd_ama} requires one more \textit{sharp} operator query of $\vb$ at Step 7. 
As  mentioned earlier, when $h_{\Vc}$ has decomposable structures, computing this sharp operator can be done efficiently (e.g., closed form or  parallel/distributed manner).

The following theorem shows the bounds on the objective residual $f(\xbar^k) - \fopt$ and the feasibility gap $\norm{\Ab\bar{\ub}^k + \Bb\bar{\vb}^k - \cb}$ of \eqref{eq:constr_cvx} at $\xbar^k$.

\begin{theorem}\label{th:primal_recover1}
Let $\set{\xbar^k}$ with $\xbar^k := (\bar{\ub}^k, \bar{\vb}^k)$ be the sequence generated by Algorithm \ref{alg:pd_ama} and $L_{d^1} := \frac{\norm{\Ab}^2}{\mu_p}$. 
Then, the following estimates hold:
\begin{align}\label{eq:ama_primal_convergence1}
{\!\!\!\!\!\!\!\!}\left\{\!\!\begin{array}{ll}
&\vert f(\bar{\xb}^k) - f^{\star} \vert \leq \max\left\{ \frac{L_{d^1}\norm{\lbd^0}^2}{\gamma(k+1)} + \gamma D_{\Uc}, \frac{2L_{d^1}\norm{\lbds}\norm{\lbd^0 \!-\! \lbds}}{\gamma(k+1)} + \norm{\lbds}\sqrt{\frac{L_{d^1}D_{\Uc}}{k+1}} \right\},\vspace{1ex}\\
&\norm{\Ab\bar{\ub}^k + \Bb\bar{\vb}^k - \cb} \leq \frac{2L_{d^1}\norm{\lbd^0 - \lbds}}{\gamma(k+1)} + \sqrt{\frac{L_{d^1}D_{\Uc}}{k+1}}.
\end{array}\right.{\!\!\!\!\!\!}
\end{align}
Consequently, if we choose $\gamma := \frac{\epsilon}{2D_{\Uc}}$, which is optimal, then the worst-case iteration-complexity of Algorithm \ref{alg:pd_ama} to achieve the $\epsilon$-solution $\xbar^k$ of \eqref{eq:constr_cvx} in the sense of Definition \ref{de:approx_sols} is $\mathcal{O}\left(\frac{L_{d^1}D_{\Uc}}{\epsilon^2}R_0^2\right)$, where $R_0 := \max\set{2, 3\norm{\lbds}, 2\norm{\lbd^0}, 2\norm{\lbd^0 - \lbds}}$.
\end{theorem}

\begin{proof}
Since $0 <  \eta_i \leq \frac{1}{L^{\gamma}_{d^1}}$ by Step 5 of Algorithm \ref{alg:pd_ama}, for any $\lbd\in\R^n$, it follows from \eqref{eq:upper_surrogate} that
\begin{align}\label{eq:th31_est1}
d_{\gamma}(\lbd^{i+1}) &\geq \ell_i^{\gamma}(\lbd) +  \frac{1}{\eta_i}\iprods{\lbd^{i+1} - \lbd^i, \lbd^i - \lbd} +  \frac{1}{2\eta_i}\norm{\lbd^{i+1} - \lbd^i}^2 \nonumber\\
&=\ell_i^{\gamma}(\lbd)  + \frac{1}{2\eta_i}\left[\norm{\lbd^{i+1} - \lbd}^2 - \norm{\lbd^i - \lbd}^2\right],
\end{align}
where $\ell^{\gamma}_i(\lbd) := d^1_{\gamma}(\lbdh^i) + \iprods{\nabla{d^1_{\gamma}}(\lbdh^i), \lbd - \lbdh^i} + d^2(\lbd^{i+1}) + \iprods{\nabla{d^2}(\lbd^{i+1}),\lbd - \lbd^{i+1}} + \iprods{\cb, \lbd}$ and $\nabla{d^2}(\lbd^{i+1})\in\partial{d^2}(\lbd^{i+1})$ is a subgradient of $d^2$ at $\lbd^{i+1}$.

Next, we consider $\ell_i^{\gamma}(\lbd)$.  
We first note that, for any $i = 0, \cdots, k$, we have
\begin{align}\label{eq:th31_est1b}
d^1_{\gamma}(\lbd^i) \!+\! \iprods{\nabla{d^1_{\gamma}}(\lbd^i), \lbd \!-\! \lbd^i} &= g(\hat{\ub}^{i\!+\!1}) \!+\! \gamma p_{\Uc}(\hat{\ub}^{i\!+\!1}) - \iprods{\Ab\hat{\ub}^{i+1}, \lbd^i} - \iprods{\Ab\hat{\ub}^{i+1}, \lbd - \lbd^i}\nonumber\\
& = g(\hat{\ub}^{i+1}) - \iprods{\Ab\hat{\ub}^{i+1}, \lbd} + \gamma p_{\Uc}(\hat{\ub}^{i+1}).
\end{align} 
Second, by Step 6 of Algorithm \ref{alg:pd_ama}, we have $\vt^i \in \vb^{\sharp}(\Bb^T\lbd^{i+1})$, which implies
\begin{align}\label{eq:th31_est1c}
d^2(\lbd^{i+1}) + \iprods{\nabla{d^2}(\lbd^{i+1}), \lbd - \lbd^{i+1}}  &= h(\tilde{\vb}^i)  - \iprods{\Bb\tilde{\vb}^i,\lbd^{i+1}} - \iprods{\Bb\tilde{\vb}^i, \lbd - \lbd^{i+1}} \nonumber\\
&= h(\tilde{\vb}^i)  - \iprods{\Bb\tilde{\vb}^i,\lbd}.
\end{align}
Summing up \eqref{eq:th31_est1b} and \eqref{eq:th31_est1c} and using the definition of $\ell_i^{\gamma}$, we obtain
\begin{align}\label{eq:th31_est2}
\ell_i^{\gamma}(\lbd) &= g(\ut^i) \!+\! h(\vt^i) \!-\! \iprods{\Ab\ut^i + \Bb\vt^i \!-\! \cb,\lbdh^i} + \iprods{\cb \!-\! \Ab\ut^i - \Bb\vt^i, \lbd - \lbdh^i} \!+\! \gamma p_{\Uc}(\ut^i) \nonumber\\
&=f(\xt^i) - \iprods{\Ab\ut^i + \Bb\vt^i - \cb,\lbd} + \gamma p_{\Uc}(\ut^i).
\end{align}
By \eqref{eq:d1_gamma_est}, we have $d_{\gamma}(\lbd) \leq d(\lbd) + \gamma D_{\Uc} \leq d^{\star} + \gamma D_{\Uc} := \bar{d}^{\star}_{\gamma}$ for any $\lbd\in\R^n$.
Substituting \eqref{eq:th31_est2} into \eqref{eq:th31_est1}, subtracting to $\bar{d}^{\star}_{\gamma}$, and summing up the result from $i=0$ to $k$, we obtain
\begin{align}\label{eq:th31_est2b}
\sum_{i=0}^k\eta_i\big[\bar{d}^{\star}_{\gamma} - d_{\gamma}(\lbdh^{i+1})\big] &\leq \sum_{i=0}^k\eta_i\big[\bar{d}^{\star}_{\gamma} - f(\xt^i) +  \iprods{\Ab\ut^i + \Bb\vt^i - \cb,\lbd} - \gamma p_{\Uc}(\ut^i)\big] \nonumber\\
&+ \frac{1}{2}\left[\norm{\lbdh^0 - \lbd}^2 - \norm{\lbdh^{k+1} - \lbd}^2\right].
\end{align}
On the one hand, we note that $d(\lbd) \leq d^{\star} = f^{\star} \leq \Lc(\xb, \lbds) = f(\xb) - \iprods{\Ab\ub + \Bb\vb - \cb, \lbds}$ for any $\lbd\in\R^n$ and $\xb\in\Xc$ due to strong duality. 
Hence, $\iprods{\Ab\ubar^k + \Bb\vbar^k - \cb, \lbds} \leq  f(\xbar^k) - d^{\star}$. Moreover, $\bar{d}^{\star}_{\gamma} - d_{\gamma}(\lbdh^{i+1}) \geq 0$.
On the other hand, using the convexity of $f$ we have $S_kf(\xbar^k) \leq \sum_{i=0}^kw_if(\xt^i)$ and $S_k\iprods{\Ab\bar{\ub}^k + \Bb\bar{\vb}^k - \cb, \lbd} = \sum_{i=0}^kw_i\iprods{\Ab\ut^i + \Bb\vt^i - \cb, \lbd}$ for $w_i := \eta_i$.
Combining these expressions into \eqref{eq:th31_est2b}, and noting that $ 0 \leq p_{\Uc}(\ut^i) \leq D_{\Uc}$, we can derive
\vspace{-0.5ex}
\begin{align*}
0 &\leq \sum_{i=0}^kw_i\big[\bar{d}^{\star}_{\gamma} - f(\xt^i) +  \iprods{\Ab\ut^i + \Bb\vt^i - \cb,\lbd} - \gamma p_{\Uc}(\ut^i)\big] + \frac{1}{2}\norm{\lbd^0 - \lbd}^2\nonumber\\
&\leq S_k\big[d^{\star} - f(\xbar^k) + \iprods{\Ab\ubar^k + \Bb\vbar^k - \cb,\lbd} + \gamma D_{\Uc}  \big] + \frac{1}{2}\norm{\lbdh^0 - \lbd}^2,
\vspace{-0.5ex}
\end{align*}
which implies
\vspace{-0.5ex}
\begin{align}\label{eq:th31_est3a}
{\!\!\!\!\!}\iprods{\Ab\ubar^k \!+\! \Bb\vbar^k \!-\! \cb, \lbds} &\leq f(\xbar^k) \!-\! d^{\star}  \leq  \iprods{\Ab\ubar^k 
\!\!+\! \Bb\vbar^k \!\!-\! \cb,\lbd}   +  \frac{1}{2S_k}\norm{\lbdh^0 \!\!-\! \lbd}^2 \!+\! \gamma D_{\Uc}.{\!\!\!}
\vspace{-0.5ex}
\end{align}
Hence, we obtain
\vspace{-0.5ex}
\begin{align}\label{eq:th31_est3}
\iprods{\Ab\ubar^k + \Bb\vbar^k - \cb, \lbds - \lbd}  - \frac{1}{2S_k}\norm{\lbdh^0 - \lbd}^2 - \gamma D_{\Uc} \leq 0,
\vspace{-0.5ex}
\end{align}
for all $\lbd\in\R^n$.
Since \eqref{eq:th31_est3} holds for all $\lbd\in\R^n$, we can show that
\vspace{-0.5ex}
\begin{align}\label{eq:th31_est4}
\max_{\lbd\in\R^n}\set{ \iprods{\Ab\ubar^k + \Bb\vbar^k - \cb, \lbds - \lbd}  - \frac{1}{2S_k}\norm{\lbdh^0 - \lbd}^2 - \gamma D_{\Uc} } \leq 0,
\vspace{-0.5ex}
\end{align}
By optimizing the left-hand side over $\lbd\in\R^n$ and using $\lbd^0 = \lbdh^0$, we obtain
\vspace{-0.5ex}
\begin{align*}
S_k\norm{\Ab\ubar^k + \Bb\vbar^k - \cb}^2 + 2\iprods{\Ab\ubar^k + \Bb\vbar^k - \cb + \rb, \lbd^0 - \lbds}  - \gamma D_{\Uc} \leq 0.
\vspace{-0.5ex}
\end{align*}
Using the Cauchy-Schwarz inequality, we have $\iprods{\Ab\ubar^k + \Bb\vbar^k - \cb , \lbd^0 - \lbds} \leq \norm{\Ab\ubar^k + \Bb\vbar^k - \cb}\norm{\lbd^0 - \lbds}$.
Hence, the last inequality leads to
\vspace{-0.5ex}
\begin{align}\label{eq:th31_est5}
\norm{\Ab\ubar^k + \Bb\vbar^k - \cb}  &\leq \frac{\norm{\lbd^0 \!-\! \lbds} + \sqrt{\norm{\lbd^0 \!-\! \lbds}^2 \!+\! \gamma S_kD_{\Uc}}}{S_k} \nonumber\\
& \leq \frac{2\norm{\lbd^0 \!-\! \lbds} }{S_k} + \sqrt{\frac{\gamma D_{\Uc}}{S_k}}.
\vspace{-0.5ex}
\end{align}
Now, since $w_i = \eta_i \geq \frac{\gamma}{L_{d^1}}$ for $i=0$ to $k$, where $L_{d^1} := \frac{\norm{\Ab}^2}{\mu_p}$. Hence, $S_k \geq \frac{\gamma(k+1)}{L_{d^1}}$. 
Substituting this bound into \eqref{eq:th31_est5}, we obtain the second inequality of \eqref{eq:ama_primal_convergence1}.

To prove the first inequality of \eqref{eq:ama_primal_convergence1}, we note from \eqref{eq:th31_est3a} and $f^{\star} = d^{\star}$ that
\vspace{-0.5ex}
\begin{align*} 
f(\xbar^k) - f^{\star}  \leq  \iprods{\Ab\ubar^k + \Bb\vbar^k - \cb,\lbd} + \frac{1}{2S_k}\norm{\lbd^0 - \lbd}^2 + \gamma D_{\Uc}.
\vspace{-0.5ex}
\end{align*}
Taking $\lbd = \boldsymbol{0}^n$ into this inequality, we get
\vspace{-0.5ex}
\begin{align*}
f(\xbar^k) - f^{\star} &\leq  \frac{1}{2S_k}\norm{\lbd^0}^2 + \gamma D_{\Uc}  \leq \frac{L_{d^1}}{\gamma(k+1)}\norm{\lbd^0}^2 + \gamma D_{\Uc}.
\vspace{-0.5ex}
\end{align*}
Combining this inequality,  \eqref{eq:lower_bound}, and the second estimate of \eqref{eq:ama_primal_convergence1}, we obtain the first estimate of \eqref{eq:ama_primal_convergence1}.

Let us choose $\gamma$ such that $\frac{2L_{d^1}r_0}{\gamma(k+1)} = \sqrt{\frac{L_{d^1}D_{\Uc}}{k+1}}$, where $r_0 := \max\set{ \norm{\lbd^0 - \lbds}, \norm{\lbd^0}}$. 
Then, $\gamma = \frac{2r_0\sqrt{L_{d^1}}}{\sqrt{D_{\Uc}(k+1)}}$. Substituting this expression into \eqref{eq:ama_primal_convergence1}, we obtain
\vspace{-0.5ex}
\begin{equation*}
\begin{cases}
\vert f(\xbar^k) - \fopt \vert &\leq \max\set{\frac{2r_0\sqrt{L_{d^1}D_{\Uc}}}{\sqrt{k+1}}, \frac{3\norm{\lbds}\sqrt{L_{d^1}D_{\Uc}}}{\sqrt{k+1}} } \leq \epsilon \\
\norm{\Ab\bar{\ub}^k + \Bb\bar{\vb}^k - \cb} &\leq \frac{3\sqrt{L_{d^1}D_{\Uc}}}{\sqrt{k+1}} \leq \epsilon.
\vspace{-0.5ex}
\end{cases}
\end{equation*} 
Consequently, we obtain the worst-case complexity of Algorithm \ref{alg:pd_ama} from the  last estimates, which is  $\mathcal{O}\left(\frac{L_{d^1}D_{\Uc}}{\epsilon^2}R_0^2\right)$, where $R_0 := \max\set{2, 3\norm{\lbds}, 2\norm{\lbd^0}, 2\norm{\lbd^0 - \lbds}}$. 
In this case, we can also show that $\gamma = \frac{\epsilon}{2D_{\Uc}}$.
\Eproof
\end{proof}

\begin{remark}\label{re:line-search}
If we apply a back-tracking line-search with a bi-section strategy on $\eta_k$, then we have $0 < \eta_k \leq \frac{2}{L^{\gamma}_{d^1}}$ at Step 5 of Algorithm \ref{alg:pd_ama}. In this case, the bounds in Theorem~\ref{th:primal_recover1} still hold with $L_{d^1} = \frac{2\norm{\Ab}^2}{\mu_p}$ instead of $L_{d^1} = \frac{\norm{\Ab}^2}{\mu_p}$.
\end{remark}

\vspace{-3.5ex}
\section{The accelerated primal-dual alternating minimization algorithm}\label{sec:ac_ama}
\vspace{-2.5ex}
In this section, we incorperate  Nesterov's accelerated step into Algorithm \ref{alg:pd_ama} as done in  \cite{Goldstein2012}, but applying to \eqref{eq:smoothed_dual_prob} to obtain a new accelerated primal-dual AMA variant. 
Clearly, this algorithm can be viewed as the FISTA scheme \cite{Beck2009} applying to the smoothed dual problem \eqref{eq:smoothed_dual_prob}.

Let $t_0 := 1$ and $\lbdh^0 := \lbd^0\in\R^n$.
The main step at the iteration $k$ of the accelerated AMA method is presented as follows:
\begin{equation}\label{eq:fama_scheme}
\left\{\begin{array}{ll}
\hat{\ub}^{k+1} &:= \displaystyle\argmin_{\ub\in\Uc}\set{ g(\ub) - \iprods{\Ab^T\lbdh^k, \ub} + \gamma p_{\Uc}(\ub) } = \nabla{g_{\gamma}^{*}}(\Ab^T\lbdh^k) ,\\
\hat{\vb}^{k+1} &:= \displaystyle\argmin_{\vb\in\Vc}\set{ h(\vb) - \iprods{\Bb^T\lbdh^k, \vb} + \frac{\eta_k}{2}\norm{\cb - \Ab\hat{\ub}^{k+1} - \Bb\vb }^2 },\\
\lbd^{k+1} &:= \lbdh^k + \eta_k\left(\cb - \Ab\hat{\ub}^{k+1} - \Bb\hat{\vb}^{k+1}\right),\vspace{0.5ex}\\
t_{k+1} &:= \frac{1}{2}\big(1 + \sqrt{1 + 4t_k^2}\big),\vspace{0.5ex}\\
\lbdh^{k+1} &:= \lbd^{k+1} + \frac{t_k-1}{t_{k+1}}\big(\lbd^{k+1} - \lbdh^k\big),
\end{array}\right.
\end{equation}
where, again, $g_{\gamma}(\cdot) := g(\cdot) + \gamma p_{\Uc}(\cdot)$.
We now combine the accelerated AMA step \eqref{eq:fama_scheme} and the weighted averaging scheme \eqref{eq:averaging_sequence} to construct a new accelerated primal-dual AMA method as presented in Algorithm \ref{alg:pd_ac_ama} below.

\begin{algorithm}[!ht]\caption{(\textit{Accelerated primal-dual  alternating minimization algorithm})}\label{alg:pd_ac_ama}
\normalsize
\begin{algorithmic}
   \STATE {\bfseries Initialization:} 
    \STATE{\hspace{2.5ex}}1. Choose $\gamma := \frac{\epsilon}{D_{\Uc}}$, and $\underline{L}$ such that $0 < \underline{L} \leq L^{\gamma}_{d^1} := \frac{\norm{\Ab}^2}{\gamma\mu_p}$.
    \STATE{\hspace{2.5ex}}2. Choose an initial point $\lbd^0 \in \R^n$. 
    \STATE{\hspace{2.5ex}}3.  Set $t_0 := 1$ and $\lbdh^0 := \lbd^0$. Set $S_{-1} := 0$, ~$\ubar^{-1} := 0$ and $\vbar^{-1} := 0$. 
    \FOR{$k := 0$ {\bfseries to} $k_{\max}$}
	\STATE 4. Compute $\ut^k = \hat{\ub}^{k+1} = \ub_{\gamma}^{\ast}(\lbdh^k)$ defined in \eqref{eq:smoothed_dual_prob}.
	\STATE 5.  Choose $\eta_k\in \left(0, \frac{1}{L^{\gamma}_{d^1}}\right]$ and compute
	\begin{equation*}
	 \hat{\vb}^{k+1} := \mathrm{arg}\displaystyle\min_{\vb\in\Vc}\set{ h(\vb) - \iprods{\Bb^T\lbdh^k, \vb} + \frac{\eta_k}{2}\norm{\Ab\ut^k + \Bb\vb - \cb}^2}.
	 \vspace{-2ex}
	 \end{equation*}
	\STATE 6.  Update $\lbd^{k+1} := \lbdh^k + \eta_k(\cb - \Ab\hat{\ub}^{k+1} - \Bb\hat{\vb}^{k+1})$.
	\STATE 7. Update $t_{k+1} := 0.5\big(1 + (1 + 4t_k^2)^{1/2}\big)$ and $\lbdh^{k+1} := \lbd^{k+1} + \frac{t_k-1}{t_{k+1}}(\lbd^{k+1} - \lbdh^k)$.
	\STATE 8.  Compute $\vt^k := \vb^{\ast}(\lbd^{k+1}) \in \vb^{\sharp}(\Bb^T\lbd^{k+1})$ defined in \eqref{eq:sharp_oper}.
	\STATE 9.  Update $S_k := S_{k-1} + w_k$, with $w_k := \eta_kt_k$, and $\tau_k := \frac{w_k}{S_k}$.
	\STATE 10. Update $\bar{\ub}^{k} := (1-\tau_k)\bar{\ub}^{k-1}  + \tau_k\tilde{\ub}^k$ and $\bar{\vb}^{k} := (1-\tau_k)\bar{\vb}^{k-1}  + \tau_k\tilde{\vb}^k$.
   \ENDFOR
\STATE {\bfseries Output:} The primal sequence $\set{\xbar^k}$ with $\xbar^k := (\ubar^k, \vbar^k)$.
\end{algorithmic}
\end{algorithm}
Similar to Algorithm \ref{alg:pd_ama}, if we know the Lipschitz constant $L^{\gamma}_{d^1}$ a priori, we can use $\eta_k := \frac{1}{L^{\gamma}_{d^1}}$.
However, we can also use a backtracking line-search to adaptively choose $\eta_k := L_k^{-1}$ such that the condition \eqref{eq:linesearch_cond} holds.
We note that the complexity-per-iteration of Algorithm \ref{alg:pd_ac_ama} essentially remains the same as in Algorithm \ref{alg:pd_ama}.

The following theorem provides the bound on the absolute objective residual and the primal feasibility gap at the iteration $\xbar^k$ for Algorithm \ref{alg:pd_ac_ama}.

\begin{theorem}\label{th:primal_convergence2}
Let $\{\xbar^k\}$ be the sequence generated by Algorithm \ref{alg:pd_ac_ama} and $L_{d^1} := \frac{\norm{\Ab}^2}{\mu_p}$.
Then, the following estimates hold:
\begin{equation}\label{eq:primal_convergence2}
\left\{{\!\!\!}\begin{array}{ll}
&\vert f(\bar{\xb}^k) - f^{\star} \vert \leq \max\left\{ \frac{2L_{d^1}\norm{\lbd^0}^2}{\gamma(k+1)(k+2)} + \gamma D_{\Uc}, \frac{8L_{d^1}\norm{\lbds}\norm{\lbd^0 - \lbds}}{\gamma(k+1)(k+2)} + \norm{\lbds}\sqrt{\frac{4L_{d^1}D_{\Uc}}{(k+1)(k+2)}} \right\},\vspace{1ex}\\
&\norm{\Ab\bar{\ub}^k + \Bb\bar{\vb}^k - \cb} \leq \frac{8L_{d^1}\norm{\lbd^0 - \lbds}}{\gamma(k+1)(k+2)} + \sqrt{\frac{4L_{d^1}D_{\Uc}}{(k+1)(k+2)}}.
\end{array}\right.
\end{equation}
Consequently, if we choose $\gamma := \frac{\epsilon}{D_{\Uc}}$, which is optimal, then the worst-case iteration-complexity of Algorithm~\ref{alg:pd_ac_ama} to achieve an $\epsilon$-solution $\xbar^k$ of \eqref{eq:constr_cvx} in the sense of Definition~\ref{de:approx_sols} is $\mathcal{O}\left(\frac{\sqrt{L_{d^1}D_{\Uc}}}{\epsilon}R_0\right)$, where 
$R_0 := \max\set{4, \frac{9}{2}\norm{\lbd^0}, \frac{9}{2}\norm{\lbd^0 - \lbds}, 4\norm{\lbds}}$.
\end{theorem}

\begin{proof}
If we define $\tau_k := \frac{1}{t_k}$, then $\tau_0 = 1$, and by Step 7 of Algorithm \ref{alg:pd_ac_ama}, one has $\tau_{k+1}^2 = (1-\tau_{k+1})\tau_k^2$. 
Moreover, if we define $\tilde{\lbd}^k := \frac{1}{\tau_k}\big(\lbdh^k - (1-\tau_k)\lbd^k\big)$, then $\tilde{\lbd}^0 = \lbdh^0 = \lbd^0$.
Using Step 7 of Algorithm \ref{alg:pd_ac_ama}, we can also derive $\tilde{\lbd}^{k+1} = \frac{1}{\tau_{k+1}}\big(\lbdh^{k+1} - (1-\tau_{k+1})\lbd^{k+1}) = \tilde{\lbd}^k + \frac{1}{\tau_k}\big(\lbd^{k+1} - \lbdh^k\big)$.

By \eqref{eq:d1_gamma_est}, we have $d_{\gamma}(\lbd) \leq d(\lbd) + \gamma D_{\Uc} \leq d^{\star} + \gamma D_{\Uc} := \bar{d}^{\star}_{\gamma}$. Hence, $\bar{d}^{\star}_{\gamma} - d_{\gamma}(\lbd) \geq 0$ for any $\lbd\in\R^n$.
For $i=0,\cdots, k$, let $\ell^{\gamma}_i(\lbd) := d^1_{\gamma}(\lbdh^i) + \iprods{\nabla{d^1_{\gamma}}(\lbdh^i), \lbd - \lbdh^i} + d^2(\lbd^{i+1}) + \iprods{\nabla{d^2}(\lbd^{i+1}),\lbd - \lbd^{i+1}} + \iprods{\cb, \lbd}$.
Then, from \eqref{eq:upper_surrogate} with $0 < \eta_i \leq \gamma L_{d^1}^{-1}$, and $\ell^{\gamma}_i(\lbd^i) = d_{\gamma}^1(\lbdh^i) + \iprods{\nabla{d_{\gamma}^1}(\lbdh^i), \lbd^i - \lbdh^i} + d^2(\lbd^{i+1}) + \iprods{\nabla{d^2}(\lbd^{i+1}),\lbd^i - \lbd^{i+1}} + \iprods{\cb, \lbd} \geq d^1_{\gamma}(\lbd^i) + d^2(\lbd^i) + \iprods{\cb, \lbd} = d_{\gamma}(\lbd^i)$, we have
\begin{align}\label{eq:lm44_est1}
\begin{array}{ll}
\bar{d}_{\gamma}^{\star} \!-\! d_{\gamma}(\lbd^{i\!+\!1}) &\leq \bar{d}_{\gamma}^{\star} \!-\! \ell^{\gamma}_i(\lbd) - \eta_i^{-1}\iprods{\lbd^{i\!+\!1} - \lbdh^i, \lbdh^i \!-\! \lbd} \!-\! \frac{1}{2\eta_i}\norm{\lbd^{i\!+\!1} \!-\! \lbdh^i}^2,\\
\bar{d}_{\gamma}^{\star} - d_{\gamma}(\lbd^{i+1}) &\leq \bar{d}_{\gamma}^{\star} \!-\! d_{\gamma}(\lbd^i) \!-\! \eta_i^{-1}\iprods{\lbd^{i\!+\!1} \!-\! \lbdh^i,  \lbdh^i \!-\! \lbd^i} \!-\! \frac{1}{2\eta_i}\norm{\lbd^{i\!+\!1} \!-\! \lbdh^i}^2.
\end{array}
\end{align}
Multiplying the first inequality of \eqref{eq:lm44_est1} by $\tau_i$ and the second one by $(1-\tau_i)$ for $\tau_i \in (0,1)$ and summing the results up, we obtain
\begin{align}\label{eq:lm44_est2}
\bar{d}_{\gamma}^{\star} - d_{\gamma}(\lbd^{i+1}) &\leq (1-\tau_i)[\bar{d}_{\gamma}^{\star} - d_{\gamma}(\lbd^i)] + \tau_i[\bar{d}_{\gamma}^{\star} - \ell^{\gamma}_i(\lbd)] \nonumber\\
&- \frac{1}{\eta_i}\iprods{\lbd^{i+1} - \lbdh^i,  \lbdh^i - (1-\tau_i)\lbd^i - \tau_i\lbd} - \frac{1}{2\eta_i}\norm{\lbd^{i+1} - \lbdh^i}_2^2 \nonumber\\
&= (1-\tau_i)\left[\bar{d}_{\gamma}^{\star} - d_{\gamma}(\lbd^i)\right] + \tau_i\left[\bar{d}_{\gamma}^{\star} - \ell^{\gamma}_i(\lbd)\right]  \nonumber\\
& + \frac{\tau_i}{2\eta_i}\left[\norm{\tilde{\lbd}^i - \lbd}^2 - \norm{\tilde{\lbd}^i + \frac{1}{\tau_i}(\lbd^{i+1} - \lbdh^i) - \lbd}^2 \right],
\end{align}
where $\tilde{\lbd}^i := \frac{1}{\tau_i}\big(\lbdh^i - (1-\tau_i)\lbd^i\big)$.
Now, let $\tilde{\lbd}^{i+1} = \tilde{\lbd}^i + \frac{1}{\tau_i}(\lbd^{i+1} - \lbdh^i)$ as stated above. 
Then, \eqref{eq:lm44_est2} leads to
\begin{align*}
\bar{d}_{\gamma}^{\star} \!-\! d_{\gamma}(\lbd^{i\!+\!1}) \leq (1 \!-\! \tau_i)\left[\bar{d}_{\gamma}^{\star} \!-\! d_{\gamma}(\lbd^i)\right] \!+\! \tau_i\left[\bar{d}_{\gamma}^{\star} \!-\! \ell_i^{\gamma}(\lbd)\right]  \!+\! \frac{\tau_i^2}{2\eta_i}\left[\norm{\tilde{\lbd}^i \!-\! \lbd}^2 - \norm{\tilde{\lbd}^{i\!+\!1} \!-\! \lbd}^2 \right].
\end{align*}
Now, since $\tau_i^2 = (1-\tau_i)\tau_{i-1}^2$ and $\eta_i \leq \eta_{i-1}$, we have $\frac{\eta_i(1-\tau_i)}{\tau_i^2} \leq \frac{\eta_{i-1}}{\tau_{i-1}^2}$.
Then, since $\bar{d}_{\gamma}^{\star} - d_{\gamma}(\lbd^i) \geq 0$, the last inequality implies
\begin{align*} 
\frac{\eta_i}{\tau_i^2}\big[\bar{d}_{\gamma}^{\star} - d_{\gamma}(\lbd^{i+1})\big] &\leq \frac{\eta_{i-1}}{\tau_{i-1}^2}\big[\bar{d}_{\gamma}^{\star} - d_{\gamma}(\lbd^i)\big] + \frac{\eta_i}{\tau_i}\big[\bar{d}_{\gamma}^{\star} - \ell^{\gamma}_i(\lbd)\big]\nonumber\\
&  + \frac{1}{2}\big[\norm{\tilde{\lbd}^i - \lbd}^2 - \norm{\tilde{\lbd}^{i+1} - \lbd}^2 \big].
\end{align*}
Summing up this inequality from $i=0$ to $k$, and using the fact that $\tau_0 = 1$, we obtain
\begin{align}\label{eq:lm44_est3}
\frac{\eta_k}{\tau_k}\left[\bar{d}_{\gamma}^{\star} - d_{\gamma}(\lbd^{k+1})\right] &\leq \frac{\eta_0(1-\tau_0)}{\tau_0^2}\left[\bar{d}_{\gamma}^{\star} - d_{\gamma}(\lbd^{k})\right] + \sum_{i=0}^k\frac{\eta_i}{\tau_i}\left[\bar{d}_{\gamma}^{\star} - \ell^{\gamma}_i(\lbd)\right] \nonumber\\
&  + \frac{1}{2}\left[\norm{\tilde{\lbd}^0 - \lbd}^2 - \norm{\tilde{\lbd}^{k+1} - \lbd}^2 \right]\nonumber\\
&\leq \sum_{i=0}^k\frac{\eta_i}{\tau_i}\left[\bar{d}_{\gamma}^{\star} - \ell^{\gamma}_i(\lbd)\right] + \frac{1}{2}\norm{\tilde{\lbd}^0 - \lbd}^2.
\end{align}
Similar to the proof of \eqref{eq:th31_est2}, we have
\begin{align*} 
\ell^{\gamma}_i(\lbd) = g(\ut^i) + h(\vt^i)  - \iprods{\Ab\ut^i + \Bb\vt^i - \cb,\lbd} + \gamma p_{\Uc}(\ut^i).
\end{align*}
Next, using the convexity of $g$ and $h$, and $p_{\Uc}(\ut^i) \geq 0$, the last inequality implies
\begin{align}\label{eq:lm44_est4}
\sum_{i=0}^k\frac{\eta_i}{\tau_i}\left[\bar{d}_{\gamma}^{\star} - \ell^{\gamma}_i(\lbd)\right] &= \sum_{i=0}^k\frac{\eta_i}{\tau_i}\left[\bar{d}_{\gamma}^{\star}- g(\ut^i) - h(\vt^i)  + \iprods{\Ab\ut^i + \Bb\vt^i - \cb,\lbd} - \gamma p_{\Uc}(\ut^i) \right] \nonumber\\
&\leq S_k\left[\bar{d}_{\gamma}^{\star} -  g(\bar{\ub}^k) - h(\bar{\vb}^k) + \iprods{\Ab\bar{\ub}^k + \Bb\bar{\vb}^k - \cb,\lbd}\right].
\end{align}
Substituting \eqref{eq:lm44_est4} into \eqref{eq:lm44_est3} and noting that $\bar{d}_{\gamma}^{\star} \geq d_{\gamma}(\lbd^{k+1})$, $f(\bar{\xb}^k) = g(\bar{\ub}^k) + h(\bar{\vb}^k)$ and $f^{\star} = d^{\star} = \bar{d}_{\gamma}^{\star} - \gamma D_{\Uc}$, we have
\begin{align}\label{eq:lm44_est5}
f(\bar{\xb}^k) - f^{\star} \leq \iprods{\Ab\bar{\ub}^k + \Bb\bar{\vb}^k - \cb,\lbd} +  \frac{1}{2S_k}\norm{\tilde{\lbd}^0 - \lbd}^2 + \gamma D_{\Uc}.
\end{align}
Moreover, we have $f^{\star} \leq \Lc(\xb, \lbd^{\star}) = f(\xb) - \iprods{\Ab\ub + \Bb\vb - \cb, \lbd^{\star}}$ for $\xb\in\Xc$. 
Substituting $\xb := \bar{\xb}^k$, $\ub :=\bar{\ub}^k$ and $\vb := \bar{\vb}^k$ into this inequality we get
\begin{align}\label{eq:lm44_est6}
f^{\star} \leq f(\bar{\xb}^k) - \iprods{\Ab\bar{\ub}^k + \Bb\bar{\vb}^k - \cb,\lbd^{\star}}.
\end{align}
Combining \eqref{eq:lm44_est5} and \eqref{eq:lm44_est6}, we obtain
\begin{align}\label{eq:lm44_est7}
\iprods{\Ab\bar{\ub}^k + \Bb\bar{\vb}^k - \cb,\lbd^{\star} - \lbd}  - \frac{1}{2S_k}\norm{\tilde{\lbd}^0 - \lbd}^2 - \gamma D_{\Uc} \leq 0,~~\forall\lbd\in\R^n.
\end{align}
Hence, by maximizing the left-hand side over $\lbd\in\R^n$, we finally get
\begin{align*}
\max_{\lbd\in\R^n}\Big\{\iprods{\Ab\bar{\ub}^k + \Bb\bar{\vb}^k - \cb,\lbd^{\star} - \lbd}  - \frac{1}{2S_k}\norm{\tilde{\lbd}^0 - \lbd}^2 - \gamma D_{\Uc} \Big\} \leq 0,
\end{align*}
Solving the maximization problem in this inequality, we can show that
\begin{equation}\label{eq:th32_est10}
\norm{\Ab\bar{\ub}^k + \Bb\bar{\vb}^k - \cb} \leq \frac{2\norm{\lbd^0 - \lbd^{\star}}}{S_k}  + \sqrt{\frac{\gamma D_{\Uc}}{S_k}}.
\end{equation}
We note that $t_k$ updated by Step 6 satisfies: $\frac{k+1}{2} \leq t_k \leq k+1$, and $0 < \eta_k \leq \gamma L_{d^1}^{-1}$. 
Hence, $S_k = \sum_{i=0}^kw_i = \sum_{i=0}^kt_i\eta_i \geq \gamma \sum_{i=0}^k\frac{i+1}{2L_{d^1}} = \frac{\gamma(k+1)(k+2)}{4L_{d^1}}$.
Using this estimate into \eqref{eq:th32_est10}, we get the second estimate of \eqref{eq:primal_convergence2}.

To prove the first estimate of  \eqref{eq:primal_convergence2}, we note from \eqref{eq:lm44_est5} with $\lbd := \boldsymbol{0}^n$ that
\begin{align*}
f(\bar{\xb}^k) - f^{\star} \leq \frac{1}{2S_k}\norm{\lbd^0}^2 + \gamma D_{\Uc} \leq \frac{2L_{d^1}}{\gamma(k+1)(k+2)}\norm{\lbd^0}^2 + \gamma D_{\Uc}.
\end{align*}
Combining this estimate, the second estimate of  \eqref{eq:primal_convergence2}, and \eqref{eq:lower_bound}, we obtain the first estimate of  \eqref{eq:primal_convergence2}.

Let us choose $\gamma > 0$ such that $\frac{8L_{d^1}r_0}{\gamma(k+1)(k+2)} = \sqrt{\frac{4L_{d^1}D_{\Uc}}{(k+1)(k+2)}}$, where $r_0 := \max\set{\norm{\lbd^0}, \norm{\lbd^0 - \lbds}}$. 
Then, $\gamma = \frac{4r_0\sqrt{L_{d^1}}}{\sqrt{D_{\Uc}(k+1)(k+2)}}$.
Substituting this $\gamma$ into \eqref{eq:primal_convergence2}, we obtain
\begin{equation*}
\left\{\begin{array}{ll}
\vert f(\bar{\xb}^k) - f^{\star} \vert &\leq  \max\set{\frac{9r_0\sqrt{L_{d^1}D_{\Uc}}}{2\sqrt{(k+1)(k+2)}}, \frac{4\norm{\lbds}\sqrt{L_{d^1}D_{\Uc}}}{\sqrt{(k+1)(k+2)}}} \leq\epsilon\\
\norm{\Ab\bar{\ub}^k + \Bb\bar{\vb}^k - \cb} &\leq \frac{4\sqrt{L_{d^1}D_{\Uc}}}{\sqrt{(k+1)(k+2)}} \leq \epsilon.
\end{array}\right.
\end{equation*}
Hence, the worst-case complexity of  Algorithm \ref{alg:pd_ac_ama} to achieve the $\epsilon$-solution $\xbar^k$ is $\mathcal{O}\left(\frac{\sqrt{L_{d^1}D_{\Uc}}}{\epsilon}R_0\right)$, where $R_0 := \max\set{4, \frac{9}{2}\norm{\lbd^0}, \frac{9}{2}\norm{\lbd^0 - \lbds}, 4\norm{\lbds}}$.
In this case, we also have $\gamma = \frac{\epsilon}{D_{\Uc}}$.
\Eproof
\end{proof}

\begin{remark}\label{re:bound_Du}
We note that the bounds in Theorems \ref{th:primal_recover1} and \ref{th:primal_convergence2} only essentially depend on the prox-diameter $D_{\Uc}$ of $\Uc$, but not of $\Vc$. Since we can exchange $g$ and $h$ in the alternating step, we can choose $\Uc$ or $\Vc$ that has smaller prox-diameter in our algorithms to smooth its corresponding objective.
\end{remark}

\vspace{-4ex}
\section{Application to strongly convex objectives}\label{sec:strong_convexity_case}
\vspace{-2.5ex}
We assume that either $g$ or $h$ is strongly convex. 
Without loss of generality, we can assume that $g$ is strongly convex with the convexity parameter $\mu_g > 0$ but $h$ remains non-strongly convex, then the dual component $d^1$ is concave and smooth. 
Its gradient $\nabla{d^1}(\lbd) = -\Ab\ub^{\ast}(\lbd)$ is Lipschitz continuous with the Lipschitz constant $L_{d^1} := \frac{\Vert\Ab\Vert^2}{\mu_g}$.
In this case, we can modified Algorithms \ref{alg:pd_ama} and \ref{alg:pd_ac_ama} at the following steps to capture this assumption.

\vspace{1ex}
\noindent\fbox{\parbox{0.98\textwidth}{
\vspace{-1ex}
\begin{itemize}
\item Step 1: Choose $\underline{L}$ such that $0 < \underline{L} \leq L_{d^1} := \frac{\Vert\Ab\Vert^2}{\mu_g}$.
\item Step 4: Compute $\ut^k = \hat{\ub}^{k+1} = \ub^{\ast}(\lbdh^k) = \ub^{\sharp}(\Ab^T\lbdh^k)$ defined by \eqref{eq:sharp_oper}.
\item Step 5: Choose $\eta_k \in (0, L_{d^1}^{-1}]$.
\end{itemize}
\vspace{-1ex}
}}
\vspace{0ex}

We call this modification the \textit{strongly convex variant} of Algorithms \ref{alg:pd_ama} and \ref{alg:pd_ac_ama}, respectively.
In this case, we obtain the following convergence result, which is a direct consequence of Theorems  \ref{th:primal_recover1} and \ref{th:primal_convergence2}.

\vspace{-1ex}
\begin{corollary}\label{co:primal_convergence1b}
Let $g$  be strongly convex with the convexity parameter $\mu_g > 0$.
Assume that $\set{\xbar^k}$ is the sequence generated by the strongly convex variant of Algorithm \ref{alg:pd_ama}. Then
\vspace{-1ex}
\begin{equation}\label{eq:primal_convergence1b}
\left\{\begin{array}{ll}
&\vert f(\xbar^k) - \fopt \vert \leq \frac{\Vert\Ab\Vert^2}{\mu_g(k+1)}\max\set{\norm{\lbd^0}^2, 2\norm{\lbds} \norm{\lbd^0 -  \lbds}},\vspace{0.75ex}\\
&\Vert\Ab\ubar^k + \Bb\vbar^k - \cb\Vert \leq \frac{2\Vert\Ab\Vert^2\norm{\lbd^0 - \lbds}}{\mu_g(k+1)}.
\end{array}\right.
\vspace{-0.5ex}
\end{equation}
Consequently,  the worst-case iteration-complexity of this variant to achieve an $\epsilon$-solution $\xbar^k$ of \eqref{eq:constr_cvx} is $\mathcal{O}\left(\frac{\Vert\Ab\Vert^2R_0}{\mu_g\epsilon}\right)$, where $R_0 := \max\set{\norm{\lbd^0}^2, 2\norm{\lbds}\norm{\lbd^0 - \lbds}}$.

Alternatively, assume that $\set{\xbar^k}$ is the sequence generated by the strongly convex variant of Algorithm \ref{alg:pd_ac_ama}. 
Then
\vspace{-1ex}
\begin{equation}\label{eq:primal_convergence1b}
\left\{\begin{array}{ll}
&\vert f(\xbar^k) - \fopt \vert \leq \frac{2\Vert\Ab\Vert^2}{\mu_g(k+1)(k+2)}\max\set{ \norm{\lbd^0}^2, 4\norm{\lbds} \norm{\lbd^0 -  \lbds} },\vspace{0.75ex}\\
&\Vert\Ab\ubar^k + \Bb\vbar^k - \cb\Vert \leq \frac{8\Vert\Ab\Vert^2\norm{\lbd^0 - \lbds}}{\mu_g(k+1)(k+2)}.
\end{array}\right.
\vspace{-0.5ex}
\end{equation}
Consequently,  the worst-case iteration-complexity of this variant to achieve an $\epsilon$-solution $\xbar^k$ of \eqref{eq:constr_cvx} is $\mathcal{O}\left(\Vert\Ab\Vert\sqrt{\frac{R_0}{\mu_g\epsilon}}\right)$, where $R_0 := \max\set{2\norm{\lbd^0}^2, 8\norm{\lbds}\norm{\lbd^0 - \lbds}}$.
\end{corollary}

\vspace{-2ex}
\begin{remark}\label{re:optimal_rate}
It is important to note that, even $h$ is not strongly convex, our accelerated primal-dual AMA algorithm still achieves the $\mathcal{O}(1/\sqrt{\epsilon})$-worst case iteration-complexity, which is different from existing dual accelerated schemes \cite{Beck2014,Necoara2008,necoara2014iteration,Polyak2013}.
In addition, if $h$ is also strongly convex, then the sharp-operator $\vb^{\sharp}(\cdot)$ of $h_{\Vc}$ is well-defined and single-valued  without requiring Assumption A.\ref{as:A1b}.
\end{remark}
\vspace{-1ex}

We note that our results present in Corollary \ref{co:primal_convergence1b} can be considered as the primal-dual variants of the AMA methods in \cite{Goldstein2012}, while the result presented in Theorems \ref{th:primal_recover1} and \ref{th:primal_convergence2} is an extension to the non-strongly convex case.

%% file: TranDinh_OptLetter2015_concl.tex
\vspace{-4ex}
\section{Concluding remarks}\label{sec:concluding}
\vspace{-2.5ex}
We have introduce a new weighted averaging scheme, and combine the AMA idea and Nesterov's smoothing technique to develop new primal-dual AMA methods, Algorithm \ref{alg:pd_ama} and Algorithm \ref{alg:pd_ac_ama}, for solving prototype constrained convex optimization problems of the form \eqref{eq:constr_cvx} without strong convexity assumption.
Then, we have incorporated  Nesterov's accelerated step into Algorithm \ref{alg:pd_ama} to improve the worst-case iteration-complexity of the primal sequence from $\mathcal{O}\left(1/\epsilon^2\right)$ (resp., $\mathcal{O}\left(1/\epsilon\right)$ to $\mathcal{O}\left(1/\epsilon\right)$ (resp., $\mathcal{O}\left(1/\sqrt{\epsilon}\right)$. Our complexity bounds are directly given for the primal objective residual and the primal feasibility gap of \eqref{eq:constr_cvx}, which are new.
Interestingly, the $\mathcal{O}\left(1/\sqrt{\epsilon}\right)$-complexity bound is archived with only the strong convexity of $g$ or $h$, but not both of them.
We will extend this idea to other splitting schemes such as alternating direction methods of multipliers and other sets of assumptions such as the H\"{o}der continuity of the dual gradient in the forthcoming work.

%% file: TranDinh_OptLetter2015_appendix.tex
\vspace{-4.5ex}
\appendix
\normalsize
\section{Appendix: The proof of Lemma \ref{le:upper_surrogate}}\label{apdx:le:upper_surrogate}
\vspace{-2.5ex}
The concavity and smoothness of $d_1^{\gamma}$ is trivial \cite{Nesterov2005c}.
In addition, the equivalence between the AMA scheme  \eqref{eq:ama_scheme}  and the forward-backward splitting method was proved in, e.g., \cite{Tseng1991,Goldstein2012}.

Let $g_{\Uc,\gamma} := g_{\gamma} + \delta_{\Uc} = g + \gamma p_{\Uc} + \delta_{\Uc}$ and $h_{\Vc} := h + \delta_{\Vc}$.
We first write the optimality condition for the two convex subproblems in \eqref{eq:ama_scheme} as
\begin{align*}
\nabla{g_{\Uc,\gamma}}(\uhat^{k\!+\!1}) - \Ab^T\lbdh^k = 0, ~\text{and}~\nabla{h_{\Vc}}(\vhat^{k\!+\!1}) - \Bb^T\lbdh^k - \eta_k\Bb^T(\cb - \Ab\uhat^{k\!+\!1} - \Bb\vhat^{k\!+\!1}).
\end{align*}
Using the third line of \eqref{eq:ama_scheme} we obtain from the last expressions that
\begin{align*}
\nabla{g_{\Uc,\gamma}}(\uhat^{k+1}) = \Ab^T\lbdh^k, ~\text{and}~\nabla{h_{\Vc}}(\vhat^{k+1}) = \Bb^T\lbd^{k+1},
\end{align*}
which are equivalent to
\begin{align*}
\uhat^{k+1}  = \nabla{g_{\Uc,\gamma}}^{*}(\Ab^T\lbdh^k), ~\text{and}~\vhat^{k+1} = \nabla{h_{\Vc}}^{*}(\Bb^T\lbd^{k+1}).
\end{align*}
Multiplying these expressions by $\Ab$ and $\Bb$, respectively, and adding them together, and then subtracting to $\cb$, we finally obtain
\begin{align}\label{eq:lm21_est1}
\eta_k^{-1}( \lbd^{k+1} - \lbdh^k) &= \cb - \Ab\uhat^{k+1} - \Bb\vhat^{k+1} \nonumber\\
& = \cb - \Ab\nabla{g_{\Uc,\gamma}}^{*}(\Ab^T\lbdh^k) - \Bb\nabla{h_{\Vc}}^{*}(\Bb^T\lbd^{k+1}).
\end{align}
Now, from the definition \eqref{eq:dual_func_di} of $d^1_{\gamma}$ and $d^2$, we have $\Ab\nabla{g_{\Uc,\gamma}}^{*}(\Ab^T\lbdh^k)  = -\nabla{d^1}(\lbdh^k)$ and $\Bb\nabla{h_{\Vc}}^{*}(\Bb^T\lbd^{k+1}) = -\nabla{d^2}(\lbd^{k+1})$. 
Substituting these relations into \eqref{eq:lm21_est1}, we get
\begin{align}\label{eq:lm21_est2}
\eta_k^{-1}( \lbd^{k+1} - \lbdh^k) = \cb + \nabla{d^1_{\gamma}}(\lbdh^k) + \nabla{d^2}(\lbd^{k+1}).
\end{align}
Next, under the condition \eqref{eq:linesearch_cond}, we can derive
\begin{align}\label{eq:lm21_est3}
d^1_{\gamma}(\lbdh^k) &+ \iprods{\nabla{d^1_{\gamma}}(\lbdh^k), \lbd - \lbdh^k} = d^1_{\gamma}(\lbdh^k) \!+\! \iprods{\nabla{d^1_{\gamma}}(\lbdh^k), \lbd^{k\!+\!1} \!-\! \lbdh^k} \!+\! \iprods{\nabla{d^1_{\gamma}}(\lbdh^k), \lbd \!-\! \lbd^{k\!+\!1}} \nonumber\\
&=Q^{\gamma}_{L_k}(\lbd^{k+1};\lbdh^k) +  \iprods{\nabla{d^1_{\gamma}}(\lbdh^k), \lbd - \lbd^{k+1}} + \frac{L_k}{2}\norm{\lbd^{k+1} - \lbdh^k}^2\nonumber\\
&\overset{\tiny\eqref{eq:linesearch_cond}}{\leq} d^1_{\gamma}(\lbd^{k+1}) +  \iprods{\nabla{d^1_{\gamma}}(\lbdh^k), \lbd - \lbd^{k+1}} + \frac{L_k}{2}\norm{\lbd^{k+1} - \lbdh^k}^2.
\end{align}
Let $\ell_k^{\gamma}(\lbd) := d^1_{\gamma}(\lbdh^k) + \iprods{\nabla{d^1_{\gamma}}(\lbdh^k), \lbd - \lbdh^k} + d^2(\lbd^{k+1}) + \iprods{\nabla{d^2}(\lbd^{k+1}),\lbd - \lbd^{k+1}} + \iprods{\cb, \lbd}$.
Using this experesion in \eqref{eq:lm21_est3}, and then combining the result with \eqref{eq:lm21_est2} and $d_{\gamma}(\cdot) = d^1_{\gamma}(\cdot) + d^2(\cdot) + \iprods{\cb,\cdot}$, we finally get
\begin{align*}
\ell_k^{\gamma}(\lbd) &\leq d_{\gamma}(\lbd^{k+1}) +  \iprods{\nabla{d^1_{\gamma}}(\lbdh^k) + \nabla{d^2}(\lbd^{k+1}) - \cb, \lbd - \lbd^{k+1}} + \frac{L_k}{2}\norm{\lbd^{k+1} - \lbdh^k}^2 \nonumber\\
&=d_{\gamma}(\lbd^{k+1}) + \iprods{\eta_k^{-1}( \lbd^{k+1} - \lbdh^k), \lbd - \lbd^{k+1}} + \frac{L_k}{2}\norm{\lbd^{k+1} - \lbdh^k}^2 \nonumber\\
&=d_{\gamma}(\lbd^{k+1}) + \iprods{\eta_k^{-1}( \lbd^{k+1} - \lbdh^k), \lbd - \lbdh^k} - \left(\frac{1}{\eta_k} - \frac{L_k}{2}\right)\norm{\lbd^{k+1} - \lbdh^k}^2,
\end{align*}
which is the first inequality of \eqref{eq:upper_surrogate}.
The second inequality of \eqref{eq:upper_surrogate} follows from the first one, $d^1_{\gamma}(\lbd^k) + \iprods{\nabla{d^1_{\gamma}}(\lbdh^k), \lbd - \lbdh^k} \geq d^1_{\gamma}(\lbd)$ and $d^2(\lbd^{k+1}) + \iprods{\nabla{d^2}(\lbd^{k+1}),\lbd - \lbd^{k+1}} \geq d^2(\lbd)$ due to the concavity of $d^1_{\gamma}$ and $d^2$, respectively.
\Eproof